\documentclass[preprint]{elsarticle}

\usepackage{amsmath,amssymb,amsthm,amscd}
\usepackage{delarray,enumerate}

\newtheorem{theorem}{Theorem}[section]
\newtheorem{lemma}[theorem]{Lemma}
\newtheorem{proposition}[theorem]{Proposition}
\newtheorem{cor}[theorem]{Corollary}
\theoremstyle{definition}
\newtheorem{definition}[theorem]{Definition}
\theoremstyle{remark}
\newtheorem{remark}[theorem]{Remark}
\newtheorem{example}[theorem]{Example}
\numberwithin{equation}{section}

\begin{document}

\begin{frontmatter}

  \title{Algebraic characterization of quasi-isometric spaces via the
    Higson compactification}

  \author[label1]{Jes\'us A. \'Alvarez L\'opez\fnref{label3}} %
  \fntext[label3]{Research of first author supported by DGICYT and
    MICINN, Grants PB95-0850 and MTM2008-02640.}
  \address[label1]{{Departamento de Xeometr\'\i a e Topolox\'\i a},
    {Facultade de Matem\'aticas}, {Universidade de Santiago de
      Compostela}, {Campus Universitario Sur}, {15706 Santiago de
      Compostela, Spain} \fnref{label2}} \ead{jesus.alvarez@usc.es}

  \author[label2]{Alberto Candel\corref{cor1}}
  \address[label2]{{Department of Mathematics}, {California State
      University at Northridge}, {Northridge, CA 91330}, {U.S.A.}}
  \cortext[cor1]{Corresponding Author} \ead{alberto.candel@csun.edu}


\begin{abstract} 
  The purpose of this article is to characterize the quasi-isometry
  type of a proper metric space via the Banach algebra of Higson
  functions on it.
\end{abstract}

\end{frontmatter}

\section{Introduction}

The Gelfand representation is a contravariant functor from the
category whose objects are commutative Banach algebras with unit and
whose morphisms are unitary homomorphisms into the category whose
objects are compact Hausdorff spaces and whose morphisms are
continuous mappings. This functor associates to a unitary Banach
algebra its set of unitary characters (they are automatically
continuous), topologized by the weak* topology.

That functor (the Gelfand representation) is right adjoint to the
functor that to a compact Hausdorff space $S$ assigns the algebra
$C(S)$ of continuous complex valued functions on $S$ normed by the
supremum norm. If $A$ is a commutative Banach algebra without radical,
then $A$ is isomorphic to an algebra of continuous complex valued
functions on the space of unitary characters of $A$.

The Gelfand representation is the base for an algebraic
characterization of compactifications of topological spaces.  A
compactification of a topological space, $X$, is a pair $(X^\kappa,
e)$ consisting of a compact Hausdorff topological space $X^\kappa$ and
an embedding $e:X \to X^\kappa$ with open dense image. (Thus, only
locally compact spaces admit compactifications in this sense.) The
complement $ X^\kappa \setminus e(X)$ is called the corona (or growth)
of the compactification $(X^\kappa,e)$, and denoted by $\kappa X$.
Usually $X$ is identified with its image $e(X)$ and thus regarded as a
subspace of $X^\kappa$.  In this case, the closure
$\overline{X}=X^\kappa$ and the boundary $\partial X= \kappa X$.

Let $X$ be a topological space and let $(X^\kappa, e)$ be a
compactification of $X$. Then $C_b(X^\kappa)=C(X^\kappa)$ is a Banach
algebra and the embedding $e:X\to X^\kappa$ induces an algebraic
isomorphism of $C_b(X^\kappa)$ into the Banach algebra $C_b(X)$ via
composition with the embedding $e$. The image of $C_b(X^\kappa)$ in
$C_b(X)$ consists precisely of all the bounded continuous functions on
$X$ that admit a continuous extension to $X^\kappa$ (via $e$). It
therefore contains the constant functions on $X$ and generates the
topology of $X$ in the sense that if $E$ is a compact subset of $X$
and $x\in X\setminus E$, then there is a function in
$e^*C_b(X^\kappa)$ that takes on the value $0$ at $x$ and is
identically $1$ on $E$.  Conversely, if $A$ is a Banach subalgebra of
$C_b(X)$ that contains the constant functions on $X$ and generates the
topology of $X$, then $A$ is isomorphic to the algebra of (bounded)
continuous functions on a compactification of $X$.

For example, $C_b(X)$, the algebra of bounded continuous functions on
$X$, corresponds to the Stone-\v{C}ech compactification of $X$, and
$\mathbb{C}+C_0(X)$, the subalgebra of $C_b(X)$ generated by the
constants and the continuous functions that vanish at infinity,
corresponds to the one-point compactification of $X$.

\begin{theorem}[Gelfand]
  Two locally compact Hausdorff spaces, $X$ and $Y$, are homeomorphic
  if and only if the Banach algebras $C_b(X)$ and $C_b(Y)$ are
  algebraically isomorphic.
\end{theorem}

In fact, an algebraic isomorphism $C_b(Y)\to C_b(X)$ induces a
homeomorphism $X^\beta\to Y^\beta$ that maps $X$ onto $Y$.

The present paper is motivated by algebraic characterizations of
topological structures for which the above theorem is a milestone.
Refinements of this milestone that motivated the present paper include
the work of Nakai~\cite{Nakai60} on the algebraic characterization of
the holomorphic and quasi conformal structures of Riemann surfaces.

Let $R$ be a Riemann surface and $M(R)$ the algebra of bounded complex
valued functions on $R$ which are absolutely continuous on lines and
have finite Dirichlet integral $D(f)= \displaystyle\int_R df\wedge
\star\, df$.  Endowed with the norm $\|f\|=\displaystyle\sup_{x\in
  R}|f(x)|+[D(f)]^{1/2}$, $M(R)$ is a commutative Banach algebra.  It
contains the constant functions and it also contains the compactly
supported smooth functions, so it generates the topology of $R$.
Therefore $M(R)$ is the algebra of continuous complex valued functions
on a compactification of $R$, the so called Royden compactification.

Work on the Royden compactification culminated in the following theorem
of Nakai~\cite{Nakai60}:

\begin{theorem}[Nakai]
  Two Riemann surfaces $R$ and $R'$ are quasi-conformally equivalent
  if and only if the corresponding algebras $M(R)$ and $M(R')$ are
  algebraically isomorphic.

  Two Riemann surfaces $R$ and $R'$ are conformally equivalent if and
  only if there is a norm preserving isomorphism between $M(R)$ and
  $M(R')$.
\end{theorem}

The Royden algebra can be defined on any locally compact metric space,
$(X,d)$, endowed with a Borel measure $\mu$. If $f$ is a complex
valued function on $X$, then its gradient norm is the function $|
\nabla f |$ on $X$ given by $|\nabla f|(x) =
\displaystyle{\limsup_{z\to x} \dfrac{|f(x)-f(z)|}{d(x,z)}}$. A
function $f$ on $X$ is a Royden function if it is bounded, continuous,
and satisfies $\int_X |\nabla f|^2 \cdot \mu<\infty$. The family of
Royden functions on $X$ form a subalgebra of the algebra of bounded
continuous functions which contains the constant functions and the
compactly supported functions.  Its completion with respect to the
norm given by $\| f\| = \displaystyle\sup_{x\in X} |f(x)| + (\int_X
|\nabla f|^2 \, \mu)^{1/2}$ is a Banach algebra and it gives rise to a
compactification of $X$, called the Royden compactification of $X$. If
$f_n$ is a Cauchy sequence of Royden functions (with respect to the
Royden norm), then $f_n$ converges uniformly to a continuous function
on $X$. Therefore, to each element in the Royden algebra, there
corresponds a bounded continuous function on $X$. This correspondence
$M(X) \to C_b(X)$ is a norm decreasing, injective, algebraic
homomorphism.

Nakai and others have studied and extended Nakai's Theorem on the
Royden algebra and Royden compactification of Riemann surfaces to
other metric spaces: Riemannian manifolds (Nakai~\cite{Nakai72},
Lelong-Ferrand~\cite{Ferrand73}), and domains in Euclidean spaces
(Lewis~\cite{Lewis71}). A generalization of Nakai's Theorem involving
Royden $p$-compactifications was also given in \cite{Nakai00}. The
following theorem is a representative result of those works.

\begin{theorem}
  Let $R$, $R'$ be Riemannian manifolds of dimension $\dim R=\dim
  R'>2$, endowed with the induced path metric structure and Riemannian
  measure. Then $R$ and $R'$ are quasi-isometrically homeomorphic if
  and only if $R$ and $R'$ have homeomorphic Royden compactifications.
\end{theorem}

In the present paper we prove an analogous theorem for metric spaces
and their coarse quasi-isometries in the sense of Gromov. The algebra
of functions that characterizes the coarse quasi-isometry type is the
Higson algebra. A Higson function (\textit{cf.}
Definition~\ref{defn:higson-function} below) on a locally compact
metric space, $(M,d)$, is a bounded Borel function, $f$, on $M$ such
that, for each $r>0$, its $r$-expansion $\nabla_r f(x)=\sup \{
|f(x)-f(y)|\mid d(x,y)\le r\}$ is in $\mathcal{B}_0(M)$, the algebra
of bounded Borel functions on $M$ that vanish at $\infty$.  The Higson
functions on $M$ form a Banach algebra, denoted by
$\mathcal{B}_\nu(M)$, and the subalgebra of continuous Higson
functions, $C_\nu(M)$, determine the so called Higson compactification
$M^\nu$ of $M$. The Higson corona (or growth) is the complement $\nu
M=M^\nu\setminus M$. Some topological properties of this Higson
compactification were studied in \cite{DranishnikovFerry97},
\cite{DKU98} and \cite{Keesling94}.

\begin{theorem}\label{t:main theorem}
  Two proper metric  spaces $(M,d)$ and $(M'd')$ are coarsely
  equivalent if and only if there is an algebraic homomorphism of
  Higson algebras $\mathcal{B}_\nu(M')\to \mathcal{B}_\nu(M)$
  that induces an isomorphism
  $C(\nu M')\to C(\nu M)$, where $\nu M$ and $\nu M'$ are the coronas
  of the respective Higson compactifications of $M$ and $M'$.
\end{theorem}

The ``only if'' part of Theorem~\ref{t:main theorem} has no version
with continuous Higson functions, which justifies the use of Borel
ones.  For instance, $\mathbf{Z}$ and $\mathbf{R}$ are coarsely
equivalent, but any continuous map $\mathbf{R}\to\mathbf{Z}$ is
constant, and therefore no homomorphism $C_\nu(\mathbf{Z})\to
C_\nu(\mathbf{R})$ induces an isomorphism $C(\nu\mathbf{Z})\to
C(\nu\mathbf{R})$.

Other geometric properties of metric spaces have been shown to have a
purely algebraic characterization; one example of such properties is
illustrated by recent work of Bourdon~\cite{Bourdon07}. To each metric
space he associates an algebra of functions based on a Besov norm, and
then he proves that two metric spaces are homeomorphic via a
quasi-Moebius homeomorphism if and only if those algebras are
isomorphic.

It appears of interest to analyze what other geometric structures on
a metric space can be inferred from naturally defined Banach algebras
of functions on it.

We thank the referee for valuable comments that helped correct the
presentation of this paper.

\section{Coarse quasi-isometries}\label{s:coarse q.-i.}

Let $(M,d)$ and $(M',d')$ be arbitrary metric spaces.  A mapping
$f:M\to M'$ is said to be {\em Lipschitz\/} if there is some $C>0$
such that
\[
d'(f(x),f(y))\le C\,d(x,y)
\]
for all $x,y\in M$. Any such constant $C$ is called a {\em Lipschitz
  distortion\/} of $f$. The map $f$ is said to be {\em bi-Lipschitz\/}
when there is some $C\ge1$ such that
$$
\frac{1}{C}\,d(x,y)\le d'(f(x),f(y))\le C\,d(x,y)
$$
for all $x,y\in M$. In this case, the constant $C$ will be called a {\em
bi-Lipschitz distortion\/} of $f$.

A family of Lipschitz maps is called {\em equi-Lipschitz\/} when all
the maps in it have some common Lipschitz distortion. A family of
bi-Lipschitz maps is said to be {\em equi-bi-Lipschitz\/} when all of
its maps have some common bi-Lipschitz distortion, which is called an
{\em equi-bi-Lipschitz distortion\/}.

A {\em net\/} \index{net} in a metric space $(M,d)$ is a subset
$A\subset M$ such that $d(x,A)\le K$ for some $K>0$ and all $x\in M$.
On the other hand, a subset $A$ of $M$ is said to be {\em separated\/}
when there is some $\delta>0$ such that $d(x,y)>\delta$ for every pair
of different points $x,y\in A$. The terms $K$-net and
$\delta$-separated net will be also used.

\begin{lemma}\label{l:separated net} Let $K>0$. There is some
  $K$-separated $K$-net in $M$. Moreover any $K$-net of $M$ contains a
  $K$-separated $2K$-net of $M$.
\end{lemma}

\begin{proof}
  Let $\mathcal S$ be the family of $K$-separated subsets of $M$. By
  using Zorn's lemma, it follows that there exists a maximal element
  $A\in{\mathcal S}$. If $d(x,A)>K$ for some $x\in M$, then
  $A\cup\{x\}\in{\mathcal S}$, contradicting the maximality of $A$.
  Hence $A$ is a $K$-net in $M$.

  Let $A$ be a $K$-net for $M$. The above shows that there is a
  $K$-separated $K$-net $B$ for the metric space $A$. It easily
  follows that $B$ is a $2K$-net for $M$.
\end{proof}

The concept of coarse quasi-isometry was introduced by
M.~Gromov~\cite{Gromov93} as follows.\footnote{It is also called {\em
    rough isometry\/} in the context of potential theory
  \cite{Kanai85}} A {\em coarse quasi-isometry\/} between metric
spaces $(M,d)$ and $(M',d')$ is a bi-Lipschitz bijection between some
nets $A\subset M$ and $A'\subset M'$; in this case, $M$ and $M'$ are
said to have the same {\em coarse quasi-isometry type\/} or to be {\em
  coarsely quasi-isometric\/}.  A coarse quasi-isometry between $M$
and itself will be called a {\em coarse quasi-isometric
  transformation\/} of $M$.  For some $K>0$ and $C\ge1$, the pair
$(K,C)$ is said to be a {\em coarsely quasi-isometric distortion\/} of
a coarse quasi-isometry if it is a bi-Lipschitz bijection between
$K$-nets with bi-Lipschitz distortion $C$. A family of {\em
  equi-coarse quasi-isometries\/} is a collection of coarse
quasi-isometries that have a common coarse distortion.

Two coarse quasi-isometries between $(M,d)$ and $(M',d')$, say $f:A\to
A'$ and $g:B\to B'$, are {\em close\/} if there are some $r,s>0$ such
that
\begin{gather*}
  d(x,B)\le r \,,\,d(y,A)\le r \;,\\
  d(x,y)\le r\Longrightarrow d'(f(x),g(y))\le s\;,
\end{gather*}
for all $x\in A$ and $y\in B$.  (Such coarse quasi-isometries $f$ and
$g$ are said to be {\em $(r,s)$-close\/}.)

It is well known that ``being coarsely quasi-isometric'' is an
equivalence relation on metric spaces. Indeed, this is a consequence
of the fact that the ``composite'' of coarse quasi-isometries makes
sense up to closeness~\cite{AlvarezLopezCandelXX}.

The equivalence classes of the closeness relation on coarse
quasi-isometries between metric spaces form a category of
isomorphisms. This is indeed the subcategory of isomorphisms of the
following larger category. For any set $S$ and a metric space $M$,
with metric $d$, two maps $f,g:S\to M$ are said to be {\em
  close\/}\footnote{This term is used in \cite{HigsonRoe00}. Other
  terms used to indicate the same property are {\em coarsely
    equivalent\/} \cite{Roe96}, {\em parallel\/} \cite{Gromov93}, {\em
    bornotopic\/} \cite{Roe93}, and {\em uniformly close\/}
  \cite{BlockWeinberger92}.} \index{close!maps} when there is some
$R>0$ such that $d(f(x),g(x))\le R$ for all $x\in S$; it may be also
said that these maps are {\em $R$-close\/}. If $(M',d')$ is another
metric space, a (not necessarily continuous) map $f:M\to M'$ is said
to be {\em large scale Lipschitz\/} \cite{Gromov93} if there are
constants $\lambda\ge1$ and $c>0$ such that
$$
d'(f(x),f(y))\le\lambda\,d(x,y)+c
$$
for all $x,y\in M$; in this case, the pair $(\lambda,c)$ will be
called a {\em large scale Lipschitz distortion\/} of $f$.  The map $f$
is said to be {\em large scale bi-Lipschitz\/} if there are constants
$\lambda\ge1$ and $c>0$ such that
\[
\frac{1}{\lambda}\,d(x,y)-c\le d'(f(x),f(y))\le\lambda\,d(x,y)+c
\]
for all $x,y\in M$; in this case, the pair $(\lambda,c)$ will be
called a {\em large scale bi-Lipschitz distortion\/} of $f$.  The map
$f$ will be called a {\em large scale Lipschitz equivalence\/} if it
is large scale Lipschitz and if there is another large scale Lipschitz
map $g:M'\to M$ so that $g\circ f$ and $f\circ g$ are close to the
identity maps on $M$ and $M'$, respectively.  In this case, if
$(\lambda,c)$ is a large scale Lipschitz distortion of $f$ and $g$,
and $g\circ f$ and $f\circ g$ are $R$-close to the identity maps for
some $R>0$, then $(\lambda,c,R)$ will be called a {\em large scale
  Lipschitz equivalence distortion\/} of $f$. A large scale Lipschitz
equivalence is easily seen to be large scale bi-Lipschitz.

It is well known that two metric spaces are coarsely quasi-isometric if and
only if they are isomorphic in the category of metric spaces and closeness
equivalence classes of large scale Lipschitz maps; this is part of the
content of the following two results, where the constants involved are
specially analyzed.

\begin{proposition}\label{p:large scale Lipschitz extensions} Let $f:A\to A'$
  be any coarse quasi-isometry between metric spaces $(M,d)$ and
  $(M',d')$ with coarse distortion $(K,C)$. Then $f$ is induced by a
  large scale Lipschitz equivalence $\varphi:M\to M'$ with large scale
  Lipschitz equivalence distortion $(C,2CK,K)$.
\end{proposition}

\begin{proof}
  For each $x\in M$ and $x'\in M'$, choose points $h(x)\in A$ and
  $h'(x)\in A'$ such that $d(x,h(x))\le K$ and $d'(x',h'(x'))\le K$.
  Moreover assume that $h(x)=x$ for all $x\in A$, and that $h'(x')=x'$
  for all $x'\in A'$.  Then $f$ and $f^{-1}$ are respectively induced
  by the maps $\varphi=f\circ h:M\to M'$ and $\psi=f^{-1}\circ h':M'\to
  M$. For all $x,y\in M$,
\begin{align*}
  d'(\varphi(x),\varphi(y))&\le C\,d(h(x),h(y))\\
  &\le C\,d(h(x),x)+C\,d(x,y)+C\,d(y,h(y))\\
  &\le C\,d(x,y)+2CK\;;
\end{align*}
furthermore
\[
d(x,\psi\circ\varphi(x))=d(x,h(x))\le K \;.
\]
Similarly,
\begin{gather*}
d(\psi(x'),\psi(y'))\le C\,d'(x',y')+2CK\;,\\
d(x',\varphi\circ\psi(x'))\le K\;,
\end{gather*}
for all $x',y'\in M'$, and the result follows.
\end{proof}

\begin{proposition}\label{p:restrictions of large scale Lipschitz maps} Let
  $\varphi:M\to M'$ be a large scale Lipschitz equivalence with large scale
  Lipschitz equivalence distortion $(\lambda,c,R)$. Then, for each
  $\varepsilon>0$, the map $\varphi$ induces a coarse quasi-isometry $f:A\to
  A'$ between $M$ and $M'$ with coarse distortion
\[
( R + 2 \lambda R + \lambda c + \lambda \varepsilon+c\,,\,
\lambda(1+\frac{2R+c}{\varepsilon}) )\;.
\]
\end{proposition}

\begin{proof}
  Let $\psi:M'\to M$ be a large scale Lipschitz map with large scale
  Lipschitz distortion $(\lambda,c)$ such that $\psi\circ\varphi$ and
  $\varphi\circ\psi$ are $R$-close to the identity maps on $M$ and $M'$,
  respectively.

  By Lemma~\ref{l:separated net}, there is a
  $(2R+c+\varepsilon)$-separated $(2R+c+\varepsilon)$-net $A$ of
  $M$. Let $A'=\varphi(A)$, and let $f:A\to A'$ denote the restriction
  of $\varphi$.  For all $x,y\in M$, the inequality
  \[
  d(x,y)\le d(x,\psi\circ\varphi(x))+d(\psi\circ\varphi(x),\psi\circ\varphi(y))
  +d(\psi\circ\varphi(y),y)
  \]
  implies
  \begin{equation}\label{e:d(x,y) le ...}
    d(x,y)\le\lambda\,d'(\varphi(x),\varphi(y))+2R+c\;.
  \end{equation}
  In particular, if $\varphi(x)=\varphi(y)$, then $d(x,y)\le2R+c$. Therefore
  $f$ is bijective because $A$ is $(2R+c)$-separated.

  For any $x'\in M'$, there is some $x\in A$ such that
  $d(x,\psi(x'))\le2R+c+\varepsilon$. Then
\begin{align*}
  d'(x',\varphi(x))&\le d'(x',\varphi\circ\psi(x'))+d'(\varphi\circ\psi(x'),\varphi(x))\\
  &\le R+\lambda\,d(\psi(x'),x)+c\\
  &\le R+2\lambda R+\lambda c+\lambda\varepsilon+c\;.
\end{align*}
So $A'$ is a $(R+2\lambda R+\lambda c+\lambda\varepsilon+c)$-net of $M'$.

Because $A$ is $(2R+c+\varepsilon)$-separated, if $x,y\in A$ are
distinct, then
\[
d'(f(x),f(y))\le \lambda\,d(x,y)+c
\le(\lambda+\frac{c}{2R+c+\varepsilon})\,d(x,y) .
\]
By the same reason and~\eqref{e:d(x,y) le ...}, it follows that
$d'(f(x),f(y))>\varepsilon/\lambda$. Hence
\[
d(x,y)\le\lambda\,d'(f(x),f(y))+2R+c
\le\lambda(1+\frac{2R+c}{\varepsilon})\,d'(f(x),f(y))
\]
again by~\eqref{e:d(x,y) le ...}, which finishes the proof.
\end{proof}

\section{Coarse structures}

The concept of coarse structure was introduced in Roe~\cite{Roe96},
and further developed in Higson-Roe~\cite{HigsonRoe00}, as a
generalization of the concept of the closeness relation on maps from a set
into a metric space. The basic definitions and results pertaining to
coarse structures are recalled presently.

\begin{definition}\label{defn:coarse-structure}
  A {\em coarse structure\/} on a set $X$ is a correspondence that
  assigns to each set $S$ an equivalence relation (called ``being {\em
    close\/}'') on the set of maps $S\to X$ such that the following
  compatibility conditions are satisfied:
\begin{enumerate}[(i)]

\item if $p,q:S\to X$ are close and $h:S'\to S$ is any map, then
  $p\circ h$ and $q\circ h$ are close;

\item if $S=S'\cup S''$ and if $p,q:S\to X$ are maps whose
  restrictions to both $S'$ and $S''$ are close, then $p$ and $q$ are
  close; and

\item all constant maps $S\to X$ are close to each other.

\end{enumerate}

  A set endowed with a coarse structure is called a {\em coarse space\/}.
\end{definition}

\begin{definition}
Let $X$ be a coarse space. A subset $E\subset X\times X$ is called
{\em controlled\/}\footnote{Controlled sets are called {\em
    entourages\/} in Roe~\cite{Roe96}.} if the restrictions to $E$ of
the two factor projections $X\times X\to X$ are close.
\end{definition}

The coarse structure of a coarse space $X$ is determined by its
controlled sets: two maps $p,q:S\to X$ are close if and only if the
image of $(p,q):S\to X\times X$ is controlled. Thus a coarse structure
can be also defined in terms of its controlled sets \cite{Roe96},
\cite{HigsonRoe00}.

A subset $B\subset X$ is called {\em bounded\/} if
$B\times B$ is controlled, equivalently, if the inclusion
mapping $B\hookrightarrow X$ is close to a constant mapping.  More generally,
a collection $\mathcal{U}$ of subsets of $X$ is said to be {\em
  uniformly bounded\/} if $\bigcup_{U\in\mathcal{U}}U\times U$ is
controlled. The coarse space $X$ is called {\em separable\/} if it has a
 countable uniformly bounded cover.

\begin{definition}\label{defn:coarse-map}
  A mapping $f:X\to X'$ between coarse spaces is called a {\em coarse
    map\/} \index{coarse!map} if
  \begin{enumerate}[(i)]

  \item\label{defn:coarse-map-1} whenever $p,q:S\to X$ are close maps,
    the composites $f\circ p,f\circ q:S\to X'$ are close maps; and

  \item if $B$ is a bounded subset of $X'$, then $f^{-1}(B)$ is
    bounded in $X$.

  \end{enumerate}
\end{definition}

Two coarse spaces, $X$ and $X'$, are \textit{coarsely equivalent} if
there are coarse mappings $f:X\to X'$ and $g:X'\to X$ such that
$f\circ g$ is close to the identity of $X'$ and $g\circ f$ is close to
the identity of $X$. In this case, $f$ (and $g$) are called coarse
equivalences.  The \textit{coarse category} is the category whose
objects are coarse spaces and whose morphisms are equivalence classes
of coarse mappings, two mappings being equivalent if they are close.

\begin{definition}\label{defn:uniformly coarse mapping} Let $X$, $X'$ be
  coarse spaces. A mapping $\varphi:X\to X'$ is \textit{uniformly coarse} if
\begin{enumerate}[(i)]
\item for every controlled set $E\subset X\times X$, the image
  $(\varphi\times \varphi)(E) \subset X'\times X'$ is controlled, and
\item for every controlled set $F\subset X'\times X'$, the preimage
  $(\varphi\times \varphi)^{-1}(F) \subset X\times X$ is controlled.
\end{enumerate}
\end{definition}

\begin{proposition}\label{prop:coarse equivalence is uniform} Let $X$ and
  $X'$ be coarse spaces, and let $\varphi:X\to X'$ and $\psi:X'\to X$
  be mappings satisfying \textup{(\ref{defn:coarse-map-1})} of
  Definition~\ref{defn:coarse-map}, and such that $\psi\circ \varphi$
  is close to the identity of $X$ and $\varphi\circ \psi$ is close to
  the identity of $X'$. Then $\varphi$ and $\psi$ are uniformly coarse
  \textup{(}and consequently $X$ and $X'$ are coarsely
  equivalent\textup{)}.
\end{proposition}
\begin{proof}
  It is plain that (i) of Definition~\ref{defn:uniformly coarse mapping} is
  equivalent to (i) of Definition~\ref{defn:coarse-map},
  and that (ii) of Definition~\ref{defn:uniformly coarse mapping}
  implies (ii) of Definition~\ref{defn:coarse-map}.

  Let $p_1$ and $p_2$ denote the projection mappings $X\times X\to X$.
  If $F\subset X'\times X'$ is controlled, then $(\psi\times
  \psi)(F)\subset X\times X$ is controlled, and so the mappings
  $p_1\circ (\psi\times \psi)$ and $p_2\circ (\psi\times \psi)$ of $F$
  into $X$ are close. Therefore, the mappings $p_1\circ
  (\psi\times\psi)\circ (\varphi\times \varphi)=p_1\circ (\psi\circ
  \varphi\times \psi\circ \varphi )$ and $ p_2\circ (\psi\times
  \psi)\circ (\varphi\times \varphi)= p_2\circ (\psi\circ
  \varphi\times \psi\circ \varphi )$ from $(\varphi\times
  \varphi)^{-1} (F)$ into $X$ are also close.  Since $\psi\circ
  \varphi$ is close to the identity on $X$, the mappings $p_1$ and
  $p_2$ from $(\varphi\times \varphi)^{-1} (F)$ into $X$ are also
  close, establishing property (ii) of Definition~\ref{defn:uniformly
    coarse mapping} for $\varphi$.
\end{proof}

\begin{definition}\label{defn:proper coarse space}
A coarse structure on a set $X$ is said to be a {\em proper coarse
  structure\/} if
\begin{enumerate}[(i)]

\item $X$ is equipped with a locally compact Hausdorff topology;

\item $X$ has a uniformly bounded open cover; and

\item every bounded subset of $X$ has compact closure.

\end{enumerate}
\end{definition}

A set equipped with a proper coarse structure will be called a {\em
  proper coarse space\/}.  Note that bounded subsets of a proper
coarse space are those subsets with compact closure.

A metric space, $(M,d)$, has a natural coarse structure, that is
defined by declaring two maps $f,g:S\to M$ (where $S$ is any set) to
be {\em close\/} when $\sup\{d(f(s),g(s))\ |\ s\in S\}<\infty$. This
closeness relation defines a coarse structure on $M$, which is called
its {\em metric\/}\footnote{This term is taken from
  \cite{HigsonRoe00}. It is also called {\em bounded\/} coarse
  structure in \cite{Roe96}} coarse structure. The terms {\em metric
  closeness\/} and {\em metric controlled set\/} can be used in this
case.  This coarse structure is proper if and only if the metric space
$M$ is proper in the sense that its closed balls are compact. In the
case of metric coarse structures, the above abstract coarse notions
have their usual meanings for metric spaces.

More generally, following Hurder~\cite{Hurder94}, a {\em coarse
  distance\/} (or {\em coarse metric\/}) on a set $X$ is a symmetric
map $d:X\times X\to[0,\infty)$ satisfying the triangle inequality; in
this case, $(X,d)$ is called a {\em coarse metric space}. Any coarse
distance defines a coarse structure in the same way as a metric does,
and will be also called a {\em metric\/} coarse structure. In this
section and the above one, all notions and properties are given for
metric spaces for simplicity, but they have obvious versions for
coarse metric spaces.

\begin{definition}\label{defn:coarse metric}
If $(M,d)$ and $(M',d')$ are metric spaces, the two conditions of
Definition~\ref{defn:coarse-map} on a map $f:M\to M'$ to be coarse can
be written as follows:
\begin{enumerate}[(i)]

\item ({\em Uniform expansiveness.\/}\footnote{This name comes
    from Roe~\cite{Roe96}. Other terms used to denote the same
    property are {\em uniformly bornologous\/} Roe~\cite{Roe93} and
    {\em coarsely Lipschitz\/}
    Block-Weinberger~\cite{BlockWeinberger92}.})  For each $R>0$ there
  is some $S>0$ such that
\[
d(x,z)\le R\Longrightarrow d'(f(x),f(z))\le S
\]
for all $x,z\in M$.

\item ({\em Metric properness.\/}\footnote{This term is used in
Roe~\cite{Roe96}}) For each bounded
subset $B\subset M'$, the inverse image
$f^{-1}(B)$ is bounded in $M$.
\end{enumerate}
\end{definition}

The last property admits a uniform version: a map $f:M\to M'$ is said
to be {\em uniformly metrically proper\/}\footnote{This term is used in
  \cite{Roe96}. Another term used to denote the same property is {\em
    effectively proper\/} \cite{BlockWeinberger92}.}  if for each
$R>0$ there is some $S>0$ so that
\[
d'(f(x),f(z)) \le R \Longrightarrow d(x,z) \le S
\]
for all $x,z$ in $M$. By using uniform metric properness instead of
metric properness, we get what is called the {\em rough category\/}.
More precisely, a map between metric spaces, $f:M\to M'$, is called a
{\em rough map\/} if it is uniformly expansive and uniformly
metrically proper; if moreover there is a rough map $g:X'\to X$ so
that the compositions $g\circ f$ and $f\circ g$ are respectively close
to the identity maps on $X$ and $X'$, then $f$ is called a {\em rough
  equivalence\/}; in this case, $X$ and $X'$ are said to be {\em
  roughly equivalent\/}\footnote{The term {\em uniform closeness\/} is
  used in \cite{BlockWeinberger92} to indicate this equivalence
  between metric spaces}.  Thus rough equivalences are the maps that
induce isomorphisms in the rough category. There are interesting
differences between the rough category and the coarse category of
metric spaces Roe~\cite{Roe96}, but the following result shows that
they have the same isomorphisms.

\begin{proposition}\label{p:uniform metric properness} Any coarse equivalence
  between metric spaces is uniformly metrically proper.  Moreover the
  definition of uniform metric properness is satisfied with constants
  that depend only on the constants involved in the definition of
  coarse equivalence.
\end{proposition}

\begin{proof}
  Let $f:M\to M'$ and $g:M'\to M$ be coarse maps so that $g\circ f$
  and $f\circ g$ are $r$-close to the identity maps on $M$ and $M'$
  for some $r>0$. Then, because $g$ is uniformly expansive, for any
  $R>0$ there is some $S>0$ such that
\[
d'(x',z') \le R\Longrightarrow d(g(x'),g(z')) \le S
\]
for all $x',z'\in M'$. It follows that when $x,z\in M$ are such that
$d'(f(x),f(z))\le R$, then
\[
d(x,z)\le d(x,g\circ f(x))+d(g\circ f(x),g\circ f(z)+d(g\circ f(z),z)
\le S+2r \,,\] which establishes that $f$ is uniformly metrically
proper.
\end{proof}

It is not possible to define ``equi-coarse maps'' or ``equi-coarse
equivalences'' between arbitrary coarse spaces, but in the case of
metric coarse structures the following related concepts can be
defined. A family of maps, $f_i:X_i\to X'_i$, $i\in\Lambda$, is said
to be a family of:
\begin{itemize}

\item {\em equi-uniformly expansive maps\/} if they satisfy the
  condition of uniform expansiveness involving the same constants;

\item {\em equi-uniformly metrically proper maps\/} if they satisfy the
  condition of uniform metric properness involving the same
  constants;

\item {\em equi-rough maps\/} if they are equi-uniformly expansive and
  equi-uniformly metrically proper; and

\item {\em equi-rough equivalences\/} if they are equi-rough, and
  there is another collection of equi-rough maps $g_i:X'_i\to X_i$,
  $i\in\Lambda$, and there is some $r>0$ so that the composites
  $g_i\circ f_i$ and $f_i\circ g_i$ are $r$-close to the identity maps
  on $X_i$ and $X'_i$, respectively, for all $i\in\Lambda$.

\end{itemize}

According to Proposition~\ref{p:uniform metric properness}, a
collection of equi-rough equivalences can be also properly called a
family of {\em equi-coarse equivalences\/}.

Gromov~\cite[Theorem 1.8.i]{Gromov99} characterizes complete path
metric spaces (that is, metric spaces where the distance between any
two points equals the infimum of the lengths of all paths joining
those two points) as those complete metric spaces, $(X,d)$, that
satisfy the following property: for all points $x,y$ in $X$ and every
$\varepsilon>0$, there is some point $z$ such that
$\max\{d(x,z),d(y,z)\}<\frac{1}{2}\,d(x,y)+\varepsilon$. This
condition can be called ``{\em approximate convexity\/}:'' a subset of
$\mathbf{R}^n$ satisfies this property (with respect to the induced
metric) if and only if it has convex closure. Gromov~\cite[Theorem
1.8.i]{Gromov99} establishes that a complete, locally compact metric
space is approximately convex if and only if it is geodesic: the
distance between any two points equals the length of some curve
joining those two points.

The following definition is a coarsely quasi-isometric version of
the above approximate convexity property.

\begin{definition}\label{d:coarsely quasi-convex} A metric space, $(M,d)$,
  is said to be {\em coarsely quasi-convex\/} if there are constants
  $a,b,c>0$ such that, for each $x,y\in M$, there is some finite
  sequence of points $x=x_0,\dots,x_n=y$ in $M$ such that
  $d(x_{k-1},x_k)\le c$ for all $k\in\{1,\dots,n\}$, and
\[
\sum_{k=1}^nd(x_{k-1},x_k)\le a\,d(x,y)+b\;.
\]
A family of metric spaces is said to be {\em equi-coarsely
  quasi-convex\/} if all of them satisfy the condition of being
coarsely quasi-convex with the same constants $a$, $b$, and $c$.
\end{definition}

\begin{remark}
  Definition~\ref{d:coarsely quasi-convex} can be compared with the
  concept of monogenic coarse space \cite{Roe03}. In the case of a
  metric coarse structure, the condition of being monogenic is
  obtained by removing the constants $a,b$ and the last inequality from
  Definition~\ref{d:coarsely quasi-convex}.
\end{remark}

A typical example of a coarsely quasi-convex space that is not
approximately convex is the set $V$ of vertices of a connected graph
$G$ with the metric $d_V$ induced by $G$.  This $V$ satisfies the
condition of being coarsely quasi-convex with constants $a=b=c=1$.
This metric on $V$ is the restriction of a metric on $G$ that can be
defined as follows. Choose any metric $d_e$ on each edge $e$ of $G$ so
that $e$ is isometric to the unit interval. Then the distance between
two points $x,y\in G$ is the minimum of the sums of the form
\[
d_e(x,v)+d_V(v,w)+d_f(w,y)\;,
\]
where $x,y$ lie in edges $e,f$, and $v,w$ are vertices of $e,f$,
respectively. Observe that $G$ is geodesic and $V$ is a $1/2$-net in
$G$. More generally, any net of a geodesic metric space is coarsely
quasi-convex. This is a particular case of the following result.

\begin{theorem}\label{t:coarsely quasi-convex} A metric space,
  $(M,d)$, is coarsely quasi-convex if and only if there exists a
  coarse quasi-isometry $f:A\to A'$ between $(M,d)$ and some geodesic
  metric space $(M',d')$. In this case, the coarsely quasi-isometric
  distortion of $f$ depends only on the constants involved in the
  condition coarse quasi-convexity satisfied by $M$, and conversely;
  equivalently, a family of metric spaces is equi-coarsely
  quasi-convex if and only if they are equi-coarsely quasi-isometric
  to geodesic metric spaces.
\end{theorem}

\begin{proof}
  Suppose that there is a coarse quasi-isometry $f:A\to A'$, with
  coarse distortion $(K,C)$, between $(M,d)$ and a geodesic metric space
  $(M',d')$. For all $x,y\in M$, there are some $\bar x,\bar y\in A$
  with $d(x,\bar x),d(y,\bar y)\le K$. Then there is some finite
  sequence $f(\bar x)=x'_0,\dots,x'_n=f(\bar y)$ in $M'$ such that
  $d'(x'_{k-1},x'_k)<1$ for all $k\in\{1,\dots,n\}$ and
\[
\sum_{k=1}^nd'(x'_{k-1},x'_k)=d'(f(\bar x),f(\bar y))\;.
\]
Moreover, we can assume that this is one of the shortest sequences
satisfying this condition. If  $d'(x'_{k-1},x'_k)<1/2$ and
$d'(x'_k,x'_{k+1})<1/2$ for some $k$, then the term $x'_k$ could be
removed from the sequence, contradicting its minimality. It follows
that $d'(x'_{k-1},x'_k)\ge1/2$ for at least $[n/2]$ indexes $k$. So
\[
(n-1)/4\le[n/2]/2\le\sum_{k=1}^nd'(x'_{k-1},x'_k)=d'(f(\bar x),f(\bar y))\;,
\]
which implies
\begin{equation}\label{e:coarsely quasi-convex} n\le4\,d'(f(\bar
  x),f(\bar y))+1\;.
\end{equation}

For each $k\in\{0,\dots,n\}$, there is some $\bar x'_k\in A'$ with
$d'(\bar x'_k,x'_k)\le K$, and let $\bar x_k=f^{-1}(\bar x'_k)$; for
simplicity, take $\bar x'_0=x'_0$ and $\bar x'_n=x'_n$, and thus $\bar
x_0=\bar x$ and $\bar x_n=\bar y$. To simplify the notation, let also
$x_0=x$, $x_n=y$, and $x_k=\bar x_k$ for $k\in\{1,\dots,n-1\}$. Then
\begin{align*}
d(x_{k-1},x_k)&\le d(x_{k-1},\bar x_{k-1})+d(\bar x_{k-1},\bar x_k)
+d(\bar x_k,x_k)\\
&\le2K+C\,d'(\bar x'_{k-1},\bar x'_k)\\
&\le2K+C\,(d'(\bar x'_{k-1},x'_{k-1})+d'(x'_{k-1},x'_k)
+d'(x'_k,\bar x'_k))\\
&\le2K+2CK+C
\end{align*}
for $k\in\{1,\dots,n\}$, and
\begin{align*}
\sum_{k=1}^nd(x_{k-1},x_k)&\le d(x,\bar x)
+\sum_{k=1}^nd(\bar x_{k-1},\bar x_k)+d(\bar y,y)\\
&\le2K+C\,\sum_{k=1}^nd'(\bar x'_{k-1},\bar x'_k)\\
&\le2K+C\,\sum_{k=1}^n(d'(\bar x'_{k-1},x'_{k-1})+d'(x'_{k-1},x'_k)
+d'(x'_k,\bar x'_k))\\
&\le2K+2CKn+C\,\sum_{k=1}^nd'(x'_{k-1},x'_k)\\
&\le2K+2CK\left(4\,d'(f(\bar x),f(\bar y))+1)\right)
+C\,d'(f(\bar x),f(\bar y))\\
&\le2K+2CK+(8CK+C)C\,d(\bar x,\bar y)\\
&\le2K+2CK+(8CK+C)C\,(d(\bar x,x)+d(x,y)+d(y,\bar y))\\
&\le2K+2CK+2(8CK+C)CK+(8CK+C)C\,d(x,y)\;,
\end{align*}
where~\eqref{e:coarsely quasi-convex} was used in the fifth inequality.
Thus the condition of Definition~\ref{d:coarsely quasi-convex} is satisfied
with $a$, $b$ and $c$ depending only on $K$ and $C$, as desired.

Assume now that $(M,d)$ satisfies the coarsely quasi-convex condition
(Definition~\ref{d:coarsely quasi-convex}) with constants $a$, $b$ and $c$.
By Lemma~\ref{l:separated net}, there is a $c$-separated $c$-net $A$
in $M$. By attaching an edge to any pair of points $x,y\in A$ with
$d(x,y)\le3c$, there results a graph $M'$ whose set of vertices is
$A$. For any $x,y\in A$, there is a finite sequence
$x_0=x, x_1, \dots,x_n=y$ in $M$ with $d(x_{k-1},x_k)\le c$ for all
$k\in\{1,\dots,n\}$, and
\[
\sum_{k=1}^nd(x_{k-1},x_k)\le a\,d(x,y)+b\;.
\]
For each $k$, take some $\bar x_k\in A$ with $d(x_k,\bar x_k)\le c$;
in particular, take $\bar x_0=x$ and $\bar x_n=y$. Then there is an
edge between each $\bar x_{k-1}$ and $\bar x_k$ because
\[
d(\bar x_{k-1},\bar x_k)\le d(\bar x_{k-1},x_{k-1})+d(x_{k-1},x_k)
+d(x_k,\bar x_k)\le3c\;.
\]
Therefore $M'$ is a connected graph. Let $d'$ denote the geodesic metric
on $M'$, defined as above, with each edge having a metric that makes
it isometric to the unit interval. Since $A$ is a $1$-net in $M'$, it
only remains to check that the identity map $(A,d)\to(A,d')$ is
bi-Lipschitz with bi-Lipschitz distortion depending only on $a$, $b$
and $c$.  Fix any pair of different points $x,y\in A$, and take a
sequence $x=\bar x_0,\dots,\bar x_n=y$ as above; after removing some
points of this sequence, if necessary, it may be assumed that $\bar
x_{k-1}\neq\bar x_k$ for all $k$. Since there is an edge between each
$\bar x_{k-1}$ and $\bar x_k$, it follows that $d'(x,y)\le n$. Since
$A$ is $c$-separated,
\[
c n\le\sum_{k=1}^nd(\bar x_{k-1},\bar x_k)\le a\,d(x,y)+b\le
\frac{ac+b}{c}\,d(x,y)\;,
\]
and so
\[
d'(x,y)\le\frac{ac+b}{c^2}\,d(x,y) .
\]
On the other hand, if $d'(x,y)=m$, then there is a sequence $x=y_0,\dots,y_m=y$
in $A$ with the property that  each pair $y_{k-1}$, $y_k$ is joined by
an edge; thus
$d(y_{k-1},y_k)\le3c$ for each $k$, and so
\[
d(x,y)\le\sum_{k=1}^md(y_{k-1},y_k)\le3cm=3c\,d'(x,y)\;. \qedhere
\]
\end{proof}

\begin{remark}
  Theorem~\ref{t:coarsely quasi-convex} is a coarsely quasi-isometric
  version of \cite[Proposition~2.57]{Roe03}, which asserts that the
  monogenic coarse structures are those that are coarsely equivalent
  to geodesic metric spaces.
\end{remark}

\begin{proposition}\label{p:coarsely Lipschitz} The following properties hold
  true:
\begin{enumerate}

\item[\textup{(i)}] Any large scale Lipschitz map between metric
  spaces is uniformly expansive; moreover, a family of equi-large
  scale Lipschitz maps between metric spaces is equi-uniformly
  expansive.

\item[\textup{(ii)}] Any large scale Lipschitz equivalence is a rough
  equivalence; moreover, a family of equi-large scale Lipschitz
  equivalences between metric spaces is a family of equi-rough
  equivalences.

\end{enumerate}
\end{proposition}

\begin{proof}
  Let $(M,d)$ and $(M',d')$ be metric spaces, and let $f:M\to M'$ be a
  large scale Lipschitz map. If $(\lambda,c)$ is a large scale
  Lipschitz distortion of $f$, then $f$ obviously satisfies the
  definition of uniform expansiveness with $S=\lambda R+c$ for each
  $R>0$.  This proves property~(i) because $S$ depends only on $R$,
  $\lambda$ and $c$.

  For (ii), suppose that $f$ is a large scale Lipschitz equivalence.
  Then there is a large scale Lipschitz map $g:M'\to M$, whose large
  scale Lipschitz distortion can be assumed to be also $(\lambda,c)$,
  such that $g\circ f$ and $f\circ g$ are $r$-close to the identity
  maps on $M$ and $M'$, for some $r>0$. Then
\begin{align*}
  d(x,y)&\le d(x,g\circ f(x))+d(g\circ f(x),g\circ f(y))+d(g\circ f(y),y)\\
  &\le\lambda\,d'(f(x),f(y))+2r\;.
\end{align*}
Hence $f$ satisfies the definition of uniform metric properness with
$S=\lambda R+2r$, for each $R>0$. This proves property~(ii) because $S$
depends only on $R$, $\lambda$ and $r$.
\end{proof}

\begin{example}\label{ex:coarsely quasi-convex} Let $\mathbf{N}^2=\{n^2\ |\
  n\in\mathbf{N}\}$ and $\mathbf{N}^3=\{n^3\ |\ n\in\mathbf{N}\}$ with
  the restriction of the Euclidean metric on $\mathbf{R}$. Suppose
  that $\mathbf{N}^2$ and $\mathbf{N}^3$ are large scale Lipschitz
  equivalent; \textit{i.e.}, there are large scale Lipschitz maps
  $f:\mathbf{N}^2\to\mathbf{N}^3$ and $g:\mathbf{N}^3\to\mathbf{N}^2$
  with large scale Lipschitz distortion $(\lambda,c)$ such that
  $g\circ f$ and $f\circ g$ are close to identity maps on
  $\mathbf{N}^2$ and $\mathbf{N}^3$. Let
  $\sigma,\tau:\mathbf{N}\to\mathbf{N}$ be the maps defined by
  $f(n^2)=\sigma(n)^3$ and $g(n^3)=\tau(n)^2$. Since
  $(n+1)^2-n^2\to\infty$ and $(n+1)^3-n^3\to\infty$ as $n\to\infty$,
  there is some $a\in\mathbf{N}$ such that $g\circ f(n^2)=n^2$ and
  $f\circ g(n^3)=n^3$ for all $n\ge a$, and so
  $\tau\circ\sigma(n)=\sigma\circ\tau(n)=n$ for $n\ge a$.  Assume for
  a while that there is some integer $b\ge a$ such that $\tau(n)\ge
  n+a+2$ for all $n\ge b$. Thus
\[
\tau(\{b,b+1,\dots\})\subset\{b+a+2,b+a+3,\dots\}
\]
and therefore
\[
\{a,a+1,\dots,b+a+1\}\subset\tau(\{0,1,\dots,b-1\})
\]
because $\{a,a+1,\dots\}\subset\tau(\mathbf{N})$. So
\[
b+1=\#\{a,a+1,\dots,b+a+1\}\le\#\tau(\{0,\dots,b-1\})\le b\;,
\]
which is a contradiction. Hence there is some sequence
$n_k\uparrow\infty$ in $\mathbf{N}$ such that $\tau(n_k)\le n_k+a+1$
for all $k$.  So
\[
|\tau(n_k)^2-n_k^2|\le(\tau(n_k)+n_k)^2\le(2n_k+a+1)^2\;.
\]
We can assume that $n_k\ge a$ for all
$k$. Then
\begin{align*}
n_k^3&=f\circ g(n_k^3)-f\circ g(n_1^3)+n_1^3\\
&\le\lambda\,|g(n_k^3)-g(n_1^3)|+c+n_1^3\\
&=\lambda\,|\tau(n_k)^2-\tau(n_1)^2|+c+n_1^3\\
&\le\lambda\,(|\tau(n_k)^2-n_k^2|
+n_k^2-n_1^2+|n_1^2-\tau(n_1)^2|)+c+n_1^3\\
&\le\lambda\,((2n_k+a+1)^2+n_k^2-n_1^2+(2n_1+a+1)^2)+c+n_1^3\\
&=5\lambda n_k^2+4\lambda(a+1)n_k+3\lambda n_1^2+4\lambda(a+1)n_1+2\lambda
(a+1)^2+c+n_1^3\;,
\end{align*}
which is a contradiction because $n_k\to\infty$ as $k\to\infty$.
Therefore there is no large scale Lipschitz equivalence between
$\mathbf{N}^2$ and $\mathbf{N}^3$, and thus these spaces are not
coarsely quasi-isometric. But they are coarsely equivalent; indeed,
the map $n^3\mapsto n^2$ is distance decreasing, and the map
$n^2\mapsto n^3$ is coarse: if $0<|n^2-m^2|<R$, then $n+m<R$ also, so
$|n^3-m^3|<S$ with $S=R^3$.
\end{example}

Example~\ref{ex:coarsely quasi-convex} shows that the converse of
Proposition~\ref{p:coarsely Lipschitz}-(2) does not hold in general.
Nevertheless, the following proposition shows that coarse equivalences
coincide with large scale Lipschitz equivalences for metric spaces
that are coarsely quasi-convex.

\begin{proposition}\label{p:coarsely quasi-convex} Any uniformly expansive
  map of a coarsely quasi-convex metric space to another metric space
  is large scale Lipschitz; moreover, a family of equi-uniformly
  expansive maps between metric spaces, whose domains are
  equi-coarsely quasi-convex, is a family of equi-large scale
  Lipschitz maps.
\end{proposition}

\begin{proof}
  Let $(M,d)$ and $(M',d')$ be metric spaces, and let $f:M\to M'$ be a
  uniformly expansive map.  Suppose that $M$ satisfies the condition
  of being coarsely quasi-convex with constants $a$, $b$, and $c$. Fix
  points $x,y\in M$, and let $x=x_0,\dots, x_n=y$ be a sequence of
  smallest length such that $d(x_{k-1}, x_k)\le c$, for $k=1, \cdots,
  n$, and
  \[\sum_{k=1}^n d(x_{k-1}, x_k) \le a\,d(x,y) + b\;.\]

  If both $d(x_{k-1},x_k) < c/2$ and $d(x_k,x_{k+1}) < c/2$ for some
  $k\in\{1,\dots,n-1\}$, then $d(x_{k-1},x_{k+1})<c$, and thus $x_k$
  could be removed from the sequence $x_0, x_1,\cdots, x_n$,
  contradicting that this was a sequence of smallest length. Hence
  there are at least $(n-1)/2$ indexes $k\in \{ 1, \cdots, n\}$ such
  that $d(x_{k-1},x_k)\ge c/2$. So
  \[
  a\,d(x,y)+b \ge
  \sum_{k=1}^n d(x_{k-1},x_k) \ge \frac{(n-1)c}{4}\;,
  \]
  or
  \begin{equation}\label{e:n<...}
    n \le \frac{4a}{c}\,d(x,y) + \frac{4b}{c} + 1\;.
  \end{equation}

  Since $f$ is uniformly expansive, there is some $S>0$ such that
  $d'(f(z),f(z'))\le S$ for all $z,z'\in M$ with $d(z,z') \le c$. So,
  by~\eqref{e:n<...},
  $$
  d'(f(x),f(y))\le\sum_{k=1}^nd'(f(x_{k-1}),f(x_k)) \le nS
  \le \frac{4aS}{c}\,d(x,y) + \frac{4bS}{c} + S ,
  $$
  which establishes that $f$ is large scale Lipschitz with large scale
  Lipschitz distortion depending only on $S$, $a$, $b$ and $c$.
\end{proof}

\begin{cor} \label{c:coarsely quasi-convex 1} Any coarse equivalence
  between coarsely quasi-convex metric spaces is a large scale
  Lipschitz equivalence; moreover, a family of equi-coarse
  equivalences between equi-coarsely quasi-convex spaces is a family
  of equi-large scale Lipschitz equivalences.
\end{cor}

\begin{proof}
  This is elementary by Proposition~\ref{p:coarsely quasi-convex}.
\end{proof}

\begin{cor}\label{c:coarsely quasi-convex 2}
  Two coarsely quasi-convex metric spaces are coarsely quasi-isometric if
  and only if they are coarsely equivalent; moreover, if $M_i$ and
  $M'_i$, $i\in\Lambda$, are families of equi-coarsely quasi-convex metric
  spaces, then all pairs $M_i$ and $M'_i$ are equi-coarsely
  quasi-isometric if and only they are equi-coarsely equivalent.
\end{cor}

\begin{proof}
  This follows from Propositions~\ref{p:large scale Lipschitz
    extensions},~\ref{p:restrictions of large scale Lipschitz maps}
  and~\ref{p:coarsely quasi-convex}, and Corollary~\ref{c:coarsely quasi-convex
    1}.
\end{proof}

\section{The Higson Compactification}

A significant example of coarse structure is induced by any
compactification\footnote{Only metrizable compactifications are
  considered in \cite{Roe93}, but this kind of coarse structure can
  be defined for arbitrary compactifications \cite{HigsonRoe00,Roe96}}
${X}^\kappa = \overline{X}$ of a topological space $X$, with corona
$\partial X=\kappa X={X}^\kappa \setminus X$. This coarse structure on $X$
is defined by declaring a subset $E\subset X\times X$ to be controlled
when
\[
\overline{E}\cap\partial(X\times X)
\subset\Delta_{\partial X}
\]
in $\overline{X}\times\overline{X}$, where
\[
\partial(X\times X)=\left(\partial X\times\overline{X}\right)\cup
\left(\overline{X}\times\partial X\right)\;.
\]
This is called the {\em topological\/}\footnote{This term is used in
  \cite{HigsonRoe00}. It is also called {\em continuously
    controlled\/} coarse structure in \cite{Roe96}} coarse structure
\index{topological coarse structure} associated to the given
compactification; it is proper if $\overline{X}$ is metrizable
\cite{Roe03}, \cite{Roe05}.

A compactification $\overline{X}$ of a proper coarse space $X$ is said
to be a {\em coarse compactification\/}, with {\em coarse corona\/}
$\partial X=\overline{X}\setminus X$, if the identity map from $X$
with its given coarse structure to $X$ endowed with the topological
coarse structure arising from $\overline{X}$ is a coarse
map. Intuitively, the slices of any controlled subset of $X\times X$
become small when approaching the boundary $\partial X$; in
particular, this holds for the sets of any uniformly bounded family in
$X$.

The structure of coarse compactifications of a proper coarse space $X$
can be described algebraically as follows. Let $\mathcal{B}(X)$ be the
Banach algebra of all bounded functions $X\to\mathbf{C}$ with the
supremum norm, and let $\mathcal{B}_0(X)$ be the Banach subalgebra of
all functions $f\in\mathcal{B}(X)$ that vanish at infinity;
\textit{i.e.}, such that, for any $\varepsilon>0$, there is some
compact subset $K\subset X$ so that $|f(x)|<\varepsilon$ for all $x\in
X\setminus K$.  For any $f\in\mathcal{B}(X)$ and every controlled
subset $E\subset X\times X$, let the {\em $E$-expansion\/} of $f$ be
the function $\nabla_Ef\in\mathcal{B}(X)$ defined by
\[
\nabla_Ef(x)= \sup\{|f(x)-f(y)|\ |\ (x,y)\in E\}\;.
\]

\begin{definition}\label{defn:higson-function}
  A function $f\in\mathcal{B}(X)$ is called a {\em Higson function\/}
  if $\nabla_Ef\in\mathcal{B}_0(X)$ for all controlled subsets
  $E\subset X\times X$.
\end{definition}

The set $\mathcal{B}_\nu(X)$ of all Higson functions on $X$ is a Banach
subalgebra of $\mathcal{B}(X)$ \cite{Roe96}, \cite{HigsonRoe00}, \cite{Roe03}.
If only bounded continuous functions are considered, then the notation
$C_b(X)$, $C_0(X)$ and $C_\nu(X)$ will be used instead of $\mathcal{B}(X)$,
$\mathcal{B}_0(X)$ and $\mathcal{B}_\nu(X)$, respectively.

The terms {\em $\overline{X}$-close\/} maps, {\em
  $\overline{X}$-controlled\/} sets and {\em $\overline{X}$-coarse\/}
compactification will be used in the case of the topological coarse
structure induced by a compactification $\overline{X}$ of a locally
compact space $X$.

The following lemma shows that Higson functions are preserved by
coarse maps.

\begin{lemma}\label{l:topological control} Let $\overline{X}$ be a
  compactification of a locally compact space $X$ with boundary
  $\partial X$. The following conditions are equivalent for any subset
  $E\subset X\times X$:
\begin{enumerate}

\item[\textup{(i)}] $E$ is $\overline{X}$-controlled.

\item[\textup{(ii)}] $\nabla_Ef\in\mathcal{B}_0(X)$ for every
  $f\in\mathcal{B}(X)$ having an extension $\bar
  f:\overline{X}\to\mathbf{C}$ that is continuous on the points of
  $\partial X$.

\item[\textup{(iii)}] $\nabla_Ef\in C_0(X)$ for every $f\in C_b(X)$
  having a continuous extension to $\overline{X}$.

\end{enumerate}
\end{lemma}

\begin{proof}
  To prove that property~(i) implies property~(ii), suppose that $E$
  is $\overline{X}$-controlled, and assume that some
  $f\in\mathcal{B}(X)$ has an extension $\bar
  f:\overline{X}\to\mathbf{C}$ that is continuous on the points of
  $\partial X$. Since the function $(x,y)\mapsto\left|\bar f(x)-\bar
    f(y)\right|$ on $\overline{X}\times\overline{X}$ vanishes on
  $\Delta_{\partial X}$ and is continuous on the points of $\partial
  X\times\partial X$, there is some open neighborhood $\Omega$ of
  $\Delta_{\partial X}$ in $\overline{X}\times\overline{X}$ such that
  $\left|\bar f(x)-\bar f(y)\right|<\varepsilon$ for all
  $(x,y)\in\Omega$.  On the other hand, since $E$ is
  $\overline{X}$-controlled, there is some open neighborhood $U$ of
  $\partial X$ in $\overline{X}$ such that
  \[
  \overline{E}\cap(U\times\overline{X})\subset\Omega.
  \]
  If $K$ is the compact set $K=X\setminus U$, then $\nabla_Ef(x)<\varepsilon$
  for all $x\in X\setminus K=X\cap U$, and so $\nabla_Ef\in \mathcal{B}_0(X)$.

  Property~(iii) is a particular case of property~(ii).

  To prove that property~(iii) implies property~(i), assume that
  $\nabla_Ef\in C_0(X)$ for all $f\in C_b(X)$ that admit a continuous
  extension to $\overline X$.  If $E$ were not
  $\overline{X}$-controlled, there would be a pair of different
  points, $x\in\partial X$ and $y\in\overline X$, such that either
  $(x,y)$ or $(y,x)$ is in $\overline E$. Since the family of
  controlled sets is invariant by transposition \cite{Roe96},
  \cite{HigsonRoe00}, \cite{Roe03}, it may be assumed that
  $(x,y)\in\overline E$. Then, for any continuous function $\bar
  f:\overline{X}\to\mathbf{C}$ with $\bar f(x)\neq\bar f(y)$, the
  restriction $f=\bar f|_X$ would satisfy
  $$
  \liminf_{z\to x}\nabla_Ef(z)\ge\left|\bar f(x)-\bar f(y)\right|>0\;,
  $$
  which would be a contradiction. Therefore $E$ is
  $\overline{X}$-controlled.
\end{proof}

The following is a direct consequence of Lemma~\ref{l:topological
  control}, which is contained in \cite[Proposition~2.39]{Roe03}.

\begin{cor}\label{c:coarse compactification} Let $\overline X$ be a
  compactification of a proper coarse space $X$ with boundary
  $\partial X$. Then the following conditions are equivalent:
\begin{enumerate}

\item[\textup{(i)}] $\overline X$ is a coarse compactification of $X$.

\item[\textup{(ii)}] $\mathcal{B}_\nu(X)$ contains every function in
  $\mathcal{B}(X)$ that admits an extension to $\overline X$ that is
  continuous on the points of $\partial X$.

\item[\textup{(iii)}] $C_\nu(X)$ contains every continuous function
  $X\to\mathbf{C}$ that extends continuously to $\overline X$.
\end{enumerate}
\end{cor}

\begin{proposition}\label{p:topological control} Let $\overline{X}$ and
  $\overline{X'}$ be compactifications of locally compact spaces $X$
  and $X'$ with boundaries $\partial X$ and $\partial X'$,
  respectively. Then the following properties hold:
\begin{enumerate}

\item[\textup{(i)}] A map $\varphi:X\to X'$ is coarse if it has an
  extension $\bar\varphi:\overline{X}\to\overline{X'}$ that is continuous
  on the points of $\partial X$ and such that $\bar\varphi(\partial
  X)\subset\partial X'$.

\item[\textup{(ii)}] Let $\varphi,\psi:X\to X'$ be maps with extensions
  $\bar\varphi,\bar\psi:\overline{X}\to\overline{X'}$ satisfying the
  conditions of property~\textup{(i)}.
  Then $\varphi$ and $\psi$ are
  $\overline{X'}$-close if and only if $\bar\varphi=\bar\psi$ on
  $\partial X$.

\end{enumerate}
\end{proposition}

\begin{proof}
  Let $\varphi:X\to X'$ be a map satisfying the conditions of
  property~(i).  If $B$ is a bounded subset of $X'$, then $B$ has
  compact closure in $X'$, and thus $\overline{B}\cap\partial
  X'=\emptyset$. So
  \[
  \bar\varphi(\overline{\varphi^{-1}(B)}\cap\partial X)
  \subset\overline{\bar\varphi\left(\varphi^{-1}(B)\right)}\cap\partial X'
  \subset \overline{B}\cap\partial X'=\emptyset
  \]
  because $\bar\varphi$ is continuous on the points of $\partial X$.  It
  follows that $\overline{\varphi^{-1}(B)}\cap\partial X=\emptyset$, and
  thus $\varphi^{-1}(B)$ has compact closure in $X$; that is,
  $\varphi^{-1}(B)$ is bounded in $X$.

  Let $E$ be a controlled subset of $X\times X$, and let $f:X'\to\mathbf{C}$
  be a bounded function that admits an extension $\bar f$ to
  $\overline{X'}$ that is continuous on the points of $\partial X'$.
  Lemma~\ref{l:topological control} implies that
  $\nabla_{(\varphi\times\varphi)(E)}f=\nabla_E(f\circ\varphi)\in\mathcal{B}_0(X)$
  because $\bar f\circ\bar\varphi$ is an extension of the function
  $f\circ\varphi$ that is continuous at the points of $\partial X$.  It
  follows that $(\varphi\times\varphi)(E)$ is a controlled subset of
  $X'\times X'$ by Lemma~\ref{l:topological control}. Therefore $\varphi$
  is a coarse map, which establishes property~(i).

  Let $\varphi,\psi:X\to X'$ be maps with extensions
  $\bar\varphi,\bar\psi:\overline{X}\to\overline{X'}$ satisfying the
  conditions of property~(i). Suppose first that $\bar\varphi=\bar\psi$
  on $\partial X$, and let $E=\{(\varphi(x),\psi(x))\ |\ x\in X\}$. Fix
  any point $(x',y')\in\overline{E}\cap\Delta_{\partial(X'\times
    X')}$.  Thus, for each neighborhood $\Omega$ of $(x',y')$ in
  $\overline E$, there is some point $x_\Omega\in X$ so that
  $(\varphi(x_\Omega),\psi(x_\Omega))\in\Omega$; such points $x_\Omega$
  form a net $(x_\Omega)$ in $X$. Suppose \textit{e.g.} that
  $x'\in\partial X'$. Then the net $(\varphi(x_\Omega))$ is unbounded in
  $X'$, and thus the net $(x_\Omega)$ is unbounded in $X$ because
  $\varphi$ is a coarse map according to property~(i). So there is an
  accumulation point $x$ of $(x_\Omega)$ in $\partial X$. Since
  $\bar\varphi$ and $\bar\psi$ are continuous at $x$, it follows that
  $\left(\bar\varphi(x),\bar\psi(x)\right)$ is an accumulation point of
  the net $(\varphi(x_\Omega),\psi(x_\Omega))$, which converges to
  $(x',y')$.  Hence
  $(x',y')=\left(\bar\varphi(x),\bar\psi(x)\right)\in\Delta_{\partial
    X'}$ because $\bar\varphi=\bar\psi$ on $\partial X$. This shows that
  $E$ is $\overline{X'}$-controlled, and thus $\varphi$ is
  $\overline{X'}$-close to $\psi$.

  Assume now that $\varphi$ is $\overline{X'}$-close to $\psi$;
  \textit{i.e.}, the set $E=\{(\varphi(x),\psi(x))\ |\ x\in X\}$ is
  $\overline{X'}$-controlled.  The conditions on $\bar\varphi$
  and $\bar\psi$ imply that
  \begin{align*}
    \left(\bar\varphi\times\bar\psi\right)(\Delta_{\partial X})
    &=\left(\bar\varphi\times\bar\psi\right)\left(\overline{\Delta_X}\cap(\partial
      X\times\partial X)\right)\\
    &\subset\overline{(\varphi\times\psi)(\Delta_X)}
    \cap(\partial X'\times\partial X')\\
    &=\overline{E}\cap(\partial X'\times\partial X')\\
    &\subset\Delta_{\partial X'}\;,
  \end{align*}
  which establishes that $\bar\varphi=\bar\psi$ on $\partial X$, and completes the
  proof of property~(ii).
\end{proof}

\begin{remark}
  The above result can be compared with the ``if'' part of
  \cite[Proposition~2.33]{Roe03}. Continuity on $X$ is not needed in
  Proposition~\ref{c:coarse compactification}, only the continuity on
  $\partial X$ is used, and properness is replaced by the condition to
  preserve the boundary of the compactifications. The reciprocal of
  property~(i) holds when the compactifications are first
  countable\footnote{Second countability is required in
    \cite[Proposition~2.33]{Roe03}, but only first countability is
    used in the proof.} \cite[Proposition~2.33]{Roe03}, and this
  assumption is necessary for the reciprocal
  \cite[Example~2.34]{Roe03}.
\end{remark}

According to Corollary~\ref{c:coarse compactification}, there is a
maximal coarse compactification $X^\nu$, which is the maximal ideal
space of $C_\nu(X)$; it is called the {\em Higson compactification\/}
of $X$, and its boundary $\nu X$ is called the {\em Higson corona\/}.
Since each Higson function on $X$ has a unique extension to $X^\nu$
that is continuous on the points of $\nu X$, there are canonical
isomorphisms
\begin{equation}\label{e:C(nu X)}
C(\nu X)\cong C_\nu(X)/C_0(X)\cong\mathcal{B}_\nu(X)/\mathcal{B}_0(X)\;.
\end{equation}
This isomorphism can be used to define the Higson boundary $\nu X$ for
any coarse space $X$ \cite{Roe03}.

For subsets $A$ of $X$ or of $X\times X$, the notation $A^\nu$ will be
used to indicate the closure of $A$ in $X^\nu$ or in $X^\nu\times
X^\nu$, respectively. The notation $\nu(X\times X)=(\nu X\times
X)\cup(X\times\nu X)$ will be also used.

The following lemma is contained in the proof of
\cite[Proposition~2.41]{Roe03}.

\begin{lemma}\label{l:coarse maps preserve Higson functions} Let $X$
  and ${X^{\prime}}$ be proper coarse spaces and let $\varphi:X\to
  {X^{\prime}}$ be a coarse map. Then:
\begin{enumerate}

\item[\textup{(i)}] $f\circ\varphi\in\mathcal{B}_0(X)$ for all
  $f\in\mathcal{B}_0({X^{\prime}})$, and

\item[\textup{(ii)}] $f\circ\varphi\in\mathcal{B}_\nu(X)$ for all
  $f\in\mathcal{B}_\nu({X^{\prime}})$.

\end{enumerate}
\end{lemma}

\begin{proposition}\label{p:extension of coarse maps} Let $X$ and
  ${X^{\prime}}$ be proper coarse spaces. Any coarse map $\varphi:X\to
  {X^{\prime}}$ has a unique extension $\bar\varphi:X^\nu\to
  {X'}^\nu$ that is continuous on the points of $\nu X$ and such
  that $\bar\varphi(\nu X)\subset\nu {X^{\prime}}$.
\end{proposition}

\begin{proof}
  According to Lemma~\ref{l:coarse maps preserve Higson functions},
  $\varphi$ induces a homomorphism
  \(\varphi^*:\mathcal{B}_\nu(X')\to\mathcal{B}_\nu(X)\) defined by
  $\varphi^*(f)=f\circ\varphi$, which maps
  $\mathcal{B}_0({X^{\prime}})$ to
  $\mathcal{B}_0(X)$. By~\eqref{e:C(nu X)}, $\varphi^*$ induces a
  homomorphism $C(\nu {X^{\prime}})\to C(\nu X)$. Then, by considering
  maximal ideal spaces, we get a map $\bar\varphi:X^\nu\to{X'}^\nu$,
  which extends $\varphi$ and maps $\nu X$ into $\nu
  {X^{\prime}}$. The continuity of $\bar\varphi$ on the points of $\nu
  X$ is a consequence of the fact that any Higson function has a
  unique extension to the Higson compactification which is continuous
  on the Higson corona.
\end{proof}

\begin{remark}
  The above result is slightly stronger than
  \cite[Proposition~2.41]{Roe03}, which only shows the continuity of
  the restriction $\bar\varphi:\nu X\to\nu X'$.
\end{remark}

Sometimes the Higson compactification can be easily determined, as
shown by the following result\footnote{The statement of
  \cite[Proposition~2.48]{Roe03} requires second countability but,
  indeed, its proof only uses first countability.}
\cite[Proposition~2.48]{Roe03}.

\begin{proposition}\label{p:Higson 1} Let $X$ be a proper coarse space with
  the topological coarse structure induced by a first-countable
  compactification $\overline{X}$ of $X$ with boundary $\partial X$.
  Then $\overline{X}$ and $X^\nu$ are equivalent compactifications of
  $X$, and thus $\partial X$ is homeomorphic to $\nu X$.
\end{proposition}


The hypothesis of this proposition, that the coarse structure be
induced by a first-countable compactification, is very strong, as the
following proposition shows.

\begin{proposition}\label{prop:g-delta-pt}
  Let $(M,d)$ be a proper metric space, and let $M^\nu$ be its
  Higson compactification. A point $p$ in $M^\nu$ is in $M$ if and
  only if the set $\{p\}$ is a $G_\delta$-set.
\end{proposition}

\begin{proof}
  The ``only if'' part is elementary. To prove the ``if'' part, let
  $p\in M^\nu$ be such that $\{p\}$ is a $G_\delta$ set. Then there is
  a sequence $(x_n)$ in $M$ that converges to $p$.  Suppose that
  $p\not\in M$; i.e., $p\in\nu M$. Passing to a subsequence if
  necessary, it may be assumed that there is a sequence of positive
  real numbers $r_n\uparrow\infty$ such that the metric balls
  $B(x_n,r_n)$ are mutually disjoint.  Let $f:M\to \mathbf{R}$ be the
  function given by
  \[
  f(x) = \left\{ \begin{array}{ll}
      \displaystyle{(-1)^n\,\frac{r_n-d(x,x_n)}{r_n}} &
      \text{$x$ in $B(x_n,r_n)$}\\
      0 & \text{otherwise.}\end{array} \right.
  \]
  Then $f$ extends to a continuous function $\bar{f}$ on $M^\nu$, and
  so $\displaystyle\lim_{n\to \infty} \bar{f}(x_n)=\bar{f}(p)$. But
  the definition of $f$ implies that $\bar{f}(x_n)=(-1)^n$, so the
  limit $\displaystyle\lim_{n\to \infty} \bar{f}(x_n)$ does not exist.
\end{proof}

\begin{remark}
\begin{enumerate}[(i)]
\item The argument of \cite[Example~2.53]{Roe03} can also be used to
  show that the point $p$ (in proof above) is in $M$.

\item $G_\delta$ properties are common in the study of the structure
  of the Stone-C\v{e}ch compactification of spaces (e.g.,
  Walker~\cite{Walker74}). The property brought to light here also
  plays a role in Nakai's work on the Royden compactification of
  Riemann surfaces~\cite{Nakai60}.
\end{enumerate}
\end{remark}

\begin{proposition}\label{cor:g-delta-pt}
  Let $(M,d)$ be a non-compact proper metric space.  Let $W\subset M$
  be a subset that contains metric balls of arbitrarily large radius.
  Then the closure of $W$ in $M^\nu$ is a neighborhood of a point in
  $\nu M$.
\end{proposition}

\begin{proof}
  If $W\subset M$ contains ball of arbitrarily large radius, then,
  because $M$ is not compact, there is a sequence, $(x_n)$, of points
  in $W$ without limit point in $M$, and a sequence of positive real
  numbers $r_n\uparrow \infty$ such that the metric balls $B(x_n,r_n)$
  are mutually disjoint and contained in $W$.  If $f$ is the function
  constructed in Proposition~\ref{prop:g-delta-pt}, then $g=|f|$
  admits a continuous extension, $\bar g$, to $M^\nu$ that satisfies
  $\bar g(p)=1$ for any $p\in \nu M$ that is an accumulation point of
  the sequence $(x_n)$.  Therefore ${\bar g}^{-1}(0,1]$ is an open
  neighborhood of $p$ contained in the closure of $W$ in $M^\nu$.
\end{proof}

The Higson compactification of a proper coarse space is defined as the
maximal ideal space of the algebra of Higson functions on the space.
The question arises whether it is possible to construct the Higson
compactification directly form the topological structure of the space,
or whether the Higson compactification is a Wallman-Frink
compactification. A Wallman-Frink compactification can be defined
using $\mathcal H$-ultrafilters, where $\mathcal H$ is the ring of
zero sets of Higson functions, topologized in a appropriate
manner. The resulting space may not be Hausdorff and has the Higson
compactification as a quotient space. Understanding the precise
relationship between the two compactifications will lead to an
intrinsic characterization of $\mathcal H$-set, toward which
Proposition~\ref{cor:g-delta-pt} is a minor contribution.

Even if the statement of Proposition~\ref{p:Higson 1} was not true
when the first-countability axiom is removed, the following result is
always true by the maximality of the Higson compactification among all
coarse compactifications.

\begin{proposition}\label{p:Higson 2} Proper topological coarse structures are
  induced by their Higson compactifications.
\end{proposition}

The following is a direct consequence of
Propositions~\ref{p:topological control},~\ref{p:extension of coarse
  maps} and~\ref{p:Higson 2}.

\begin{cor}\label{c:Higson 2-1} Let $X$ and ${X^{\prime}}$ be proper
  topological coarse spaces. Then the following properties hold:
\begin{enumerate}

\item[\textup{(i)}] A map $\varphi:X\to {X^{\prime}}$ is coarse
  if and only if it has an extension $\varphi^\nu:X^\nu\to
  {X^{\prime}}^{\nu}$ that is continuous on the points of $\nu X$ and
  such that $\varphi^\nu(\nu X)\subset\nu {X^{\prime}}$.

\item[\textup{(ii)}] Two coarse maps $\varphi,\psi:X\to
  {X^{\prime}}$ are close if and only if the extensions $\varphi^\nu$ and
  $\psi^\nu$, given by property~\textup{(i)}, are equal on
  $\nu X$.

\end{enumerate}
\end{cor}

The following result shows that proper metric coarse structures are
particular cases of the topological ones (Roe~\cite[Proposition~2.47]{Roe03}).

\begin{proposition}\label{p:each metric control is a topological control}
  The metric coarse structure of a proper metric space is equal to
  the topological coarse structure induced by its Higson
  compactification.
\end{proposition}

Proposition~\ref{p:each metric control is a topological control} and
Corollary~\ref{c:Higson 2-1} have the following consequences.

\begin{theorem}\label{t:Higson, metric} Let $X$ and ${X^{\prime}}$ be
  proper metric spaces. Then a map $\varphi:X\to {X^{\prime}}$
  is a coarse equivalence if and only if it has an extension
  $\varphi^\nu:X^\nu\to {X^{\prime}}^{\nu}$ such that $\varphi^\nu(\nu
  X)\subset\nu {X^{\prime}}$, $\varphi^\nu$ is continuous on the points
  of $\nu X$ and the restriction $\varphi^\nu:\nu X\to\nu {X^{\prime}}$
  is a bijection.
\end{theorem}

\begin{proof}
  The ``if'' part follows from Corollary~\ref{c:Higson 2-1}.  To prove
  the ``only if'' part, assume that $\varphi:X\to X^{\prime}$ is coarse
  and admits an extension ${\varphi}^{\nu}: X^\nu\to X^{\prime \nu}$
  that is continuous on the points of $\nu X$ and takes
  $\nu X$ bijectively onto $\nu X'$ (hence $\varphi^\nu$ induces a
  homeomorphism of $\nu X$ onto $\nu X'$ because $\nu X$ is compact
  and Hausdorff).

  The hypotheses imply that $\varphi$ is uniformly metrically
  proper. Indeed, if that was not the case, there would be a positive
  number $R>0$ and two sequences $(x_n)$ and $(z_n)$ in $X$ such that
  $d'(\varphi(x_n),\varphi(z_n))\le R$ but $d(x_n, z_n)\ge n$ for all
  $n$.  Because $\varphi:X\to X^\prime$ is coarse (metric proper and
  uniformly expansive), it may be assumed, after passing to
  subsequences if needed, that neither of the sequences $(x_n)$ and
  $(z_n)$ has accumulation points in $X$, and that neither
  $(\varphi(x_n))$ nor $(\varphi(x_n))$ have accumulation points in
  $X^\prime$. Because $d(x_n,z_n)\ge n$, an application of
  Proposition~\ref{cor:g-delta-pt} shows that the set of accumulation
  points of the sequence $(x_n)$, say $P$, and of the sequence
  $(z_n)$, say $Q$, are disjoint closed subsets of $\nu X$.  Being
  continuous on $\nu X$, the mapping ${\varphi}^\nu$ takes $P$ and $Q$
  to the set of accumulation points of the corresponding sequences
  $(\varphi(x_n))$ and $(\varphi(x_n))$, respectively. But, since
  $d'(\varphi(x_n),\varphi(z_n))\le R$, it follows that
  $\varphi^\nu(P)=\varphi^\nu(Q)$. This contradicts that $\varphi^\nu$
  induces a homeomorphism of the compact Hausdorff space $\nu X$ onto
  $\nu X'$.

  It is also true that there is an $N>0$ such that the image
  $\varphi(X)$ is $N$-dense in $X^\prime$.  For if not there would be
  a sequence $(x'_n)$ in $X^\prime$ such that the union, $W$, of the
  metric balls $B(x'_n, n)$ is disjoint from the image
  $\varphi(X)$. By Corollary~\ref{cor:g-delta-pt}, the closure of $W$
  in ${X'}^\nu$ is a neighborhood of a point $p$ in $\nu X^\prime$.
  This clearly contradicts the hypothesis that the mapping $\varphi$
  admits an extension to $X^\prime$ that takes $\nu X$ onto $\nu
  X^{\prime}$ and is continuous at the points of $\nu X^{\prime}$.

  Thus, by the above, there is some number $N$ such that $\varphi(X)$
  is $N$-dense in $X^\prime$, and so a mapping $\psi:X^\prime \to X$
  can be defined, by choosing, for each $x'$ in $X^{\prime}$ a point
  $\psi(x')$ in $X$ such that $\varphi(\psi(x'))$ is in $B(x',N)$.

  The map $\psi$ is uniformly expansive (Definition~\ref{defn:coarse
    metric}~(i)). Let $R>0$ and let $x'$ and $z'$ in $X^\prime $ be
  such that $d'(x',z')\le R$. Then, by the definition of $\psi$, the
  points $\psi(x')$ and $\psi(z')$ satisfy $d'(\varphi(\psi(x')),
  \varphi(\psi(z'))) \le d'(x',z') + 2N\le R+2N$. Because $\varphi$ is
  uniformly metrically proper, given $R+2N>0$, there is $S=S(R+2N)>0$
  such that if $d'(\varphi(x), \varphi(z))\le R+2N$, then $d(x,z)\le
  S$. This applies in particular to $x=\psi(x')$ and $z=\psi(z')$.

  The map $\psi$ is metrically proper (Definition~\ref{defn:coarse
    metric}~(ii)). Let $B\subset X$ be bounded and suppose that
  $\psi^{-1}(B)\subset X'$ is not bounded. Then there is a sequence
  $x'_1, x'_2, \cdots$ in $\psi^{-1}(B)$ such that $d'(x'_{n+1},
  x'_1)\ge n $ for all $n$. The points $x_n=\psi(x'_n)$ are all in
  $B$, and satisfy $d'(\varphi(x_n), x'_n)\le N$ for all $n$, by the
  construction of $\psi$. Therefore $d'(\varphi(x_1),\varphi(x_n)) \ge
  d(x'_1, x'_n) -2N \ge n-2N$, so that the sequence $\varphi(x_n)$ is
  unbounded in $X'$. Since all the $x_n$ are in the bounded set $B$,
  this contradicts the uniform expansiveness of $\varphi$.

  The composite mapping $\varphi\circ \psi:X^\prime \to X^\prime$ is
  close to the identity of $X^\prime$ because for any $x'$ in
  $X^\prime$, the point $\psi(x')$ is such that $d'(\varphi(\psi(x')),
  x')\le N$.

  The composite mapping $\psi\circ \varphi:X\to X$ is close to the
  identity on $X$. Indeed, by the definition of $\psi$, for any $x$ in
  $X$, the point $\psi(\varphi(x))$ is such that
  $d'(\varphi(\psi(\varphi(x))), \varphi(x))\le N$. Because $\varphi$
  is uniformly metrically proper, there is an $S=S(N)$ such that
  $d(\psi(\varphi(x)), x)\le S$ for all $x$ in $X$.
\end{proof}

According to Corollary~\ref{c:coarsely quasi-convex 2}, in the case of
coarsely quasi-convex metric spaces, the property ``coarse
equivalence'' in this statement can be replaced by the property
``coarse quasi-isometry.''

\begin{theorem}
  Let $(M,d)$ and $({M^{\prime}},d^{\prime})$ be proper metric spaces,
  and suppose that $\varphi:C_\nu({M^{\prime}})\to C_\nu(M)$ is an
  algebraic isomorphism of their associated Higson algebras. Then $M$
  and ${M^{\prime}}$ are coarsely equivalent.  Furthermore, if $M$ and
  ${M^{\prime}}$ are coarsely quasi-convex, then $\varphi$ induces a
  coarse quasi-isometry between $M$ and ${M^{\prime}}$.
 \end{theorem}

\begin{proof}
  The algebra $C_\nu(M)$ has trivial radical because it is a Banach
  subalgebra of $C_b(M)$. Therefore, by Gelfand~et~al.~\cite[Theorem 2
  of \S 9]{GelfandRaikovShilov64}, an algebraic isomorphism
  $\varphi:C_\nu({M^{\prime}})\to C_\nu(M)$ is automatically continuous
  and induces a homeomorphism of Higson compactifications $F:M^\nu\to
  {{M^{\prime}}}^\nu$.

  That homeomorphism $F$ must send $M$ homeomorphically onto $M'$
  because $M$ is first countable but no point in the Higson corona of
  $M$ is a $G_\delta$-set (Proposition~\ref{prop:g-delta-pt}).  Then
  the induced map $F:M\to {M^{\prime}}$ is a coarse equivalence by
  Theorem~\ref{t:Higson, metric}.  The last part of the statement now
  follows by invoking Proposition~\ref{p:coarsely quasi-convex}, which
  shows that a coarse mapping between coarsely quasi-convex spaces is
  a coarse quasi-isometry.
\end{proof}

We now prove a slightly strengthened version of Theorem~\ref{t:main
theorem} stated in the introduction.

\begin{theorem}
  Two proper coarse metric spaces, $(M,d)$ and $(M',d')$, are coarsely
  equivalent if and only if there is an algebraic isomorphism $C(\nu
  M')\to C(\nu M)$ induced by a homomorphism $\mathcal{B}_\nu(M')\to
  \mathcal{B}_\nu(M)$. Furthermore, if $M$ and $M'$ are defined by
  coarsely quasi-convex metrics {\rm(}or coarse metrics{\rm)} $d$ and
  $d'$, then the above condition is equivalent to the existence of a
  coarse quasi-isometry between $(M,d)$ and $(M',d')$.
\end{theorem}

\begin{proof}
  Let $\varphi:\mathcal{B}_\nu(M')\to \mathcal{B}_\nu(M)$ be an algebraic
  homomorphism inducing an algebraic isomorphism $C(\nu M')\to C(\nu
  M)$.

  Fix $K>0$ and apply Lemma~\ref{l:separated net}, to obtain
  $K$-separated $K$-nets $A\subset M$ and $A'\subset M'$. The
  inclusion mapping $A\to M$ induces a norm-decreasing algebraic
  homomorphism $\mathcal{B}_\nu(M) \to C_\nu(A)$.

  There is a Borel partition of $M'$ of the form $\{F_x\mid x\in A'\}$
  with $x\in F_x\subset B(x,K)$ for each $x\in A'$. Such a partition
  can be constructed by induction on $n$ for an enumeration $(x_n)$ of
  the points of $A'$: take $F_{x_0}=B(x_0,K)$, and
  \[
  F_{x_{n+1}}=\{x_{n+1}\}\cup(B(x_{n+1},K)\setminus(F_{x_0}\cup\dots\cup
  F_{x_n}))
  \]
  if $F_{x_0},\dots,F_{x_n}$ are constructed. Let $\chi_x$ denote the
  characteristic function of $F_x$ for each $x\in A'$. Given a
  function $f$ on $A'$, $Pf=\sum_{x\in A'} f(x)\chi_x$ is a Higson
  function on $M'$ by the argument of Roe
  in~\cite[Proposition~(5.10)]{Roe93}. This defines a homomorphism of
  algebras $P:C_\nu(A') \to \mathcal{B}_\nu(M')$ because the sets
  $F_x$ form a partition. Moreover the composition of $P$ with the
  restriction homomorphism $\mathcal{B}_\nu(M')\to C_\nu(A')$ is the
  identity on $C_\nu(A')$ because $x\in F_x$ for all $x\in A'$.
  
  It follows from the above that there is a homomorphism of algebras
  $C_\nu(A') \to C_\nu(A)$ that induces the original isomorphism
  $C(\nu A')=C(\nu M')\to C(\nu A)=C(\nu M)$.  Since $C(\nu A)=C_\nu (A)
  /C_0(A)$ and $C(\nu A)=C_\nu (A) /C_0(A)$, this homomorphism of
  algebras induces a continuous mapping $\varphi^{\nu}:A^\nu\to A^{\prime
    \nu}$ that sends $\nu A$ into $\nu A'$ homeomorphically, and such that
  the restriction $\varphi=\varphi^\nu |A$ sends $A$ into $A'$. It thus
  follows from Theorem~\ref{t:Higson, metric} that $\varphi$ induces a
  coarse equivalence $M\to M'$.

  If the metrics $d$ and $d'$ are coarsely quasi-convex, then
  $\varphi$ can be improved to a coarse quasi-isometry, because of
  Corollary~\ref{c:coarsely quasi-convex 2}.
\end{proof}

\begin{example}
  This result implies that a coarse equivalence between two locally
  compact metric spaces induces a homeomorphism between their
  respective coronas. The converse is not true in general. It follows from
  Example~\ref{ex:coarsely quasi-convex} that the Higson
  compactifications of the subspaces $\mathbf{N}^2$ and $\mathbf{N}^3$
  of the natural numbers are the same as their Stone-\v{C}ech
  compactifications (which are homeomorphic to the Stone-\v{C}ech
  compactification of the natural numbers).

  If the Continuum Hypothesis is accepted, then the Stone-\v{C}ech
  corona of the natural numbers has $2^c$
  automorphisms~(Walker~\cite{Walker74}). On the other hand, there are
  at most $c$ maps of $\mathbf{N}$ into $\mathbf{N}$. Therefore, many
  homeomorphisms of the Higson corona of $\mathbf{N}^2$ into that of
  $\mathbf{N}^3$ are not induced by a map of $\mathbf{N}^2$ into
  $\mathbf{N}^3$.
\end{example}

\bibliographystyle{model1b-num-names}

\end{document}